\documentclass[12pt,reqno]{amsart} 
\usepackage{amsmath,amsfonts,amscd,amssymb,latexsym,mathtools}
\usepackage[all,matrix,arrow,tips,curve]{xy}
\usepackage[english]{babel}
\usepackage{mathrsfs}
\usepackage{enumitem}
\usepackage{slashed}

\usepackage[left=2.9cm,right=2.9cm,top=2.66cm,bottom=2.66cm]{geometry}

\numberwithin{equation}{section}

\newtheorem{theorem}{Theorem}[section]
\newtheorem{proposition}[theorem]{Proposition}

\newtheorem{corollary}[theorem]{Corollary}

\theoremstyle{definition}
\newtheorem{example}[theorem]{Example}
\newtheorem{definition}[theorem]{Definition}
\newtheorem{remark}[theorem]{Remark}

\usepackage[sorted-cites]{amsrefs}
\usepackage[colorlinks=true, allcolors=blue]{hyperref}
\usepackage[capitalise]{cleveref} 

\newcommand{\torus}{T} 
\newcommand{\sph}{\mathbf{S}}

\newcommand{\Sec}{\operatorname{Sec}}
\newcommand{\Ric}{\operatorname{Ric}}
\newcommand{\scal}{\operatorname{scal}}
\newcommand{\scalM}{\operatorname{scal}_{g_M}}

\newcommand{\supp}{\operatorname{supp}}

\newcommand{\Kcw}{ K\textrm{-cw}_{2} }
\newcommand{\Acw}{ \widehat{A}\textrm{-cw}_{2} }
\newcommand{\Af}{ \widehat{\textbf{A}} }
\newcommand{\ch}{\textrm{ch}}

\newcommand{\ind}{{\rm ind}}
\newcommand{\sflow}{{\rm sf}}
\newcommand{\indrel}{\textrm{ind-rel}}
\newcommand{\sfrel}{\textrm{sf-rel}}
\newcommand{\glcw}{\widehat{A}^+\textrm{-cw}_2}   
\newcommand{\D}{\textrm{d}}

\makeatletter
\@namedef{subjclassname@2020}{\textup{2020} Mathematics Subject Classification}
\makeatother

\begin{document}

\title[Scalar curvature, bottom spectrum and relative $\widehat{A}$-cowaist]
      {A sharp inequality relating scalar curvature, bottom spectrum, and the relative $\widehat{A}$-cowaist on complete manifolds}

\author[Daoqiang Liu]{Daoqiang Liu}
\address{Chern Institute of Mathematics and LPMC, Nankai University, Tianjin 300071, China}
\email{\href{mailto:dqliu@nankai.edu.cn}{dqliu@nankai.edu.cn}}
\urladdr{\href{https://www.dqliu.cn}{www.dqliu.cn}}


\begin{abstract}
We introduce a refined notion of relative $\widehat{A}$-cowaist, extending the framework of Cecchini--Zeidler to complete manifolds, which may be noncompact and have compact boundary.
We establish a sharp inequality relating this invariant to the scalar curvature and the bottom spectrum of the Laplacian.
As a consequence, we partially answer a question raised by Shi, and we obtain, in all dimensions, sharp upper bounds for the bottom spectrum under scalar curvature lower bounds. In particular, this removes the bounded geometry assumption from a result of Wang--Zhu. Our approach is based on deformed Dirac operators.
\end{abstract}

\maketitle

\section{Introduction}\label{sec:intro}

The existence of Riemannian metrics with positive scalar curvature is a central problem in differential geometry, with deep connections to topology, index theory, and global analysis. A fundamental challenge is to understand how scalar curvature interacts with large-scale invariants, particularly on complete (noncompact) manifolds, where both topology and spectral behavior become subtle.

A classical approach to scalar curvature originates from the Dirac operator. In his seminal work, Lichnerowicz~\cite{Lic63} showed that a closed spin manifold admitting a metric of positive scalar curvature must have vanishing $\widehat{A}$-genus. This idea was further developed by Gromov--Lawson~\cite{GL83}, who introduced enlargeability and used index-theoretic methods to establish strong obstructions. For a comprehensive overview, see Gromov's \emph{Four lectures on scalar curvature}~\cite{Gro23}.

A complementary direction concerns the relationship between scalar curvature and spectral invariants.
For a complete Riemannian manifold $(M,g_M)$, the
Laplacian $\Delta$ is self-adjoint (see~\cite{Gaffney54}). The spectrum $\sigma(M):=\sigma(-\Delta)$ of $M$ is a closed subset of $[0,\infty)$. The bottom spectrum is defined by
\[
\lambda_1(M,g_M):=\inf \sigma(M),
\]
and it admits the variational characterization
\begin{equation}\label{defn:bottom_spectrum}
    \lambda_1(M,g_M)=\inf_{u \in C_c^{\infty}(M)\setminus \{0\} } \dfrac{\int_M |{\rm d} u|^2_{g_M} dV_{g_M}}{\int_M u^2 dV_{g_M}},
\end{equation}
where $C_c^{\infty}(M)$ denotes the space of smooth functions with compact support.
This quantity encodes important large-scale geometric information about $M$, and its interplay with scalar curvature remains a subtle problem, especially in the noncompact setting.

On closed manifolds, a striking result of Davaux~\cite{Dav03} establishes a sharp inequality relating scalar curvature, the bottom spectrum, and topological data encoded in the $K$-cowaist:
\begin{theorem}[\cite{Dav03}]\label{thm:Dav_K_cowaist_inequality}
    Let $(M,g_M)$ be an $m$-dimensional closed Riemannian spin manifold. Then
    \[
          \inf_M \scalM + \frac{4m}{m-1}\lambda_1(M,g_M) \leq \frac{c(m)}{\Kcw(M)},
    \]
    where $\scalM$ denotes the scalar curvature of $g_M$, $c(m)>0$ is a constant depending only on $m$, and $\Kcw(M)$ is the $K$-cowaist (for its definition, see \cite{Gro96}*{p.~21} for even $m$ and \cite{Shi25+}*{p.~24} for odd $m$).
\end{theorem}

This inequality relies crucially on compactness, and no comparable sharp inequality was previously available in the complete setting. In this paper, we establish such an inequality on complete (possibly noncompact) manifolds, using deformed Dirac (Callias) operators to overcome the lack of compactness naturally.

Before stating our main result, we introduce the relevant notation and definitions.
Let $M$ be an $m$-dimensional Riemannian manifold, possibly noncompact and with compact boundary; write $\operatorname{int}(M)$ for its interior.
Recall that a \textit{Gromov--Lawson pair} (abbreviated as \textit{GL-pair}) on $M$ consists of two Hermitian vector bundles $(\mathcal{E},\mathcal{F})$ equipped with metric connections, together with a parallel unitary isomorphism
\[
\mathcal{I} \colon \mathcal{E}|_{M \setminus K} \longrightarrow \mathcal{F}|_{M \setminus K},
\]
which extends smoothly to a neighborhood of $\overline{M\setminus K}$, for a compact subset $K \subset \operatorname{int}(M)$ (see \cite{GL83} and \cite{CZ24}). The set $K$ is called the \textit{support} of the pair.

A Hermitian vector bundle $\mathcal{E}$ over $M$ is called \textit{compatible} if it is a trivial bundle with trivial connection at infinity and near the boundary. 
A smooth map $\rho_M: M\to U(l)$ is called \textit{compatible} if it is trivial, i.e., locally constant at infinity and near the boundary. A compatible GL-pair $(\mathcal{E},\mathcal{F})$ together with a compatible map $\rho_M\in C^{\infty}(M,U(l))$ is called a \textit{compatible GL-triple} over $M$, denoted by $(\mathcal{E},\mathcal{F},\rho_M)$.

\begin{definition}
    Let $M$ be an $m$-dimensional Riemannian manifold, possibly noncompact and with compact boundary.
    \begin{enumerate}
        \item Assume that $m$ is even. A compatible GL-pair $(\mathcal{E},\mathcal{F})$ over $M$ is called an \textit{\(\widehat{A}\)-admissible GL-pair} if
        \[
        \int_M \Af(M) \wedge (\ch(\mathcal{E})-\ch(\mathcal{F})) \neq 0,
        \]
        where $\Af(M)$ denotes the $\widehat{A}$-form of $M$ and $\ch(\mathcal{E})$ denotes the Chern character form of $\mathcal{E}$.
        \item Assume that $m$ is odd. A compatible GL-pair $(\mathcal{E},\mathcal{F})$ together with a compatible map $\rho_M$ is called an \textit{\(\widehat{A}\)-admissible GL-triple}, denoted by $(\mathcal{E},\mathcal{F},\rho_M)$, if
        \[
        \int_M \Af(M) \wedge (\ch(\mathcal{E})-\ch(\mathcal{F})) \wedge \ch(\rho_M) \neq 0,
        \]
        where $\ch(\rho_M)$ is the odd Chern character form (cf.~\cite{Ge93}, see also \cite{Zh01}*{(1.50)}).
    \end{enumerate}
\end{definition}

For a Hermitian bundle $\mathcal{E}$, let $R^{\mathcal{E}}$ denote the curvature tensor of its connection. We define its norm by
\[
\| R^{\mathcal{E}} \|_{\infty} 
:= \sup_{p\in M} \sup_{ \substack{v_1,v_2\in T_p M\\ |v_1|=|v_2|=1} } |R^{\mathcal{E}}(v_1\wedge v_2)|,
\]
where $|R^{\mathcal{E}}(v_1\wedge v_2)|$ denotes the operator norm of the endomorphism $R^{\mathcal{E}}(v_1\wedge v_2)$.
For a smooth map \(\rho_M\in C^{\infty}(M,U(l))\), the associated curvature is
\begin{equation}\label{eq:defn_rho_curvature}
   R^{\rho_M}:=\frac{1}{4}(\rho_M^{-1}(\D\rho_M))^2, 
\end{equation}
and its norm \(\|R^{\rho_M}\|_{\infty}\) is defined analogously.

To formulate our result, we introduce a refined notion of relative $\widehat{A}$-cowaist, extending the original definition of Cecchini--Zeidler in~\cite{CZ24}*{Definition~6.5} for even-dimensional compact manifolds to arbitrary dimensions and to the noncompact setting, including manifolds with compact boundary.

\begin{definition}
    Let $(M,g_M)$ be an $m$-dimensional Riemannian manifold, possibly noncompact and with compact boundary. The \textit{relative \(\widehat{A}\)-cowaist} of $M$ is defined as
    \[
        \glcw(M):=
        \begin{cases}
               (\inf\limits_{(\mathcal{E},\mathcal{F})} \|R^{\mathcal{E}\oplus \mathcal{F} }\|_{\infty})^{-1}, & \text{if } m \text{ is even}, \\
               (\inf\limits_{(\mathcal{E},\mathcal{F},\rho_M)} \{ \|R^{\mathcal{E} \oplus \mathcal{F} }\|_{\infty} +\|R^{\rho_M}\|_{\infty} \})^{-1}, & \text{if } m \text{ is odd},  
        \end{cases}
    \]
    where $(\mathcal{E},\mathcal{F})$ ranges over all $\widehat{A}$-admissible GL-pairs and $(\mathcal{E},\mathcal{F},\rho_M)$ ranges over all $\widehat{A}$-admissible GL-triples. 
\end{definition}

We are now in a position to state our first main result.

\begin{theorem}\label{mthm:codimension_zero}
    Let $(M,g_M)$ be an $m$-dimensional complete (possibly noncompact) Riemannian spin manifold.
    Then
    \[
        \inf_M \scalM + \frac{4m}{m-1} \lambda_1(M,g_M) \leq \frac{2m(m-1)}{\glcw(M)}.
    \]
\end{theorem}

From~\cite{Wang23},~\cite{BH24} and~\cite{Shi25+}, we know that $\Kcw(M)\leq c(m)\Acw(M)$, where $\Acw(M)$ denotes the $\widehat{A}$-cowaist of $M$ and $c(m)$ is a constant depending only on $m$. Moreover, it is clear that $\Acw(M)\leq \glcw(M)$. Thus, our result extends Davaux's inequality to the noncompact setting, providing a sharp link between scalar curvature, spectral invariants, and topological data.

\begin{remark}\label{rem:spectral_cowaist}
    For an $m$-dimensional connected oriented (not necessarily complete) Riemannian manifold $(M,g_M)$, the \textit{\(c\)-spectral constant} introduced in~\cite{HKKZ24} is defined, for any $c\in\mathbf{R}$, by
    \[
        \Lambda_{c}(g_M)=
        \inf\Bigl\{ \int_M \bigl(|\mathrm{d}u|^2 + c\,\scalM\,u^2\bigr)\,dV
            \colon u\in H_0^1(M),\ \int_M u^2\,dV =1 \Bigr\},
    \]
    where $H_0^1(M)$ is the completion of $C_c^\infty(M)$ in the Sobolev $H^1$-norm. Under the assumptions of Theorem~\ref{mthm:codimension_zero}, the argument in the proof also gives
    \begin{equation}\label{eq:spectral_cowaist}
        c^{-1}\,\Lambda_{c}(g_M) \leq \frac{2m(m-1)}{\glcw(M)},
        \quad \text{for } c\in \big(\frac{m-1}{4m},\infty\big).
    \end{equation}
\end{remark}

As a consequence of Theorem~\ref{mthm:codimension_zero} and Remark~\ref{rem:spectral_cowaist}, we have the following obstruction results.

\begin{corollary}\label{cor:complete_manifold_NNSC_mean_convex_boundary}
    Let $M$ be an $m$-dimensional (possibly noncompact) spin manifold. Then $M$ does not admit a complete Riemannian metric satisfying any of the following:
    \begin{enumerate}
        \item uniformly positive scalar curvature and infinite relative $\widehat{A}$-cowaist;
        \item nonnegative scalar curvature, positive bottom spectrum and infinite relative $\widehat{A}$-cowaist;
        \item a positive $c$-spectral constant and infinite relative $\widehat{A}$-cowaist, where $c \in(\frac{m-1}{4m},\infty)$.
    \end{enumerate}
\end{corollary}

\begin{remark}
    Our result strengthens the estimates of Gromov~\cite{Gro96}*{p.~30}, B\"{a}r--Hanke~\cite{BH23}*{Theorem~2.19}, and Shi~\cite{Shi25+}*{Theorems~1.8 and 4.1}. When the bottom spectrum is positive, it removes the uniform positivity condition assumed in those works, thereby partially answering a question raised by Shi in \cite{Shi25+}*{Remark~4.3}.
\end{remark}

The following example, inspired by~\cite{Liu25a+}*{Example~0.3}, demonstrates that rigidity fails in general: within the class of warped product manifolds of infinite relative $\widehat{A}$-cowaist, there exist distinct metrics satisfying $\scal \geq -m(m-1)$ and $\lambda_1 = \frac{(m-1)^2}{4}$.

\begin{example}\label{example:mthm_codimension_zero}
    Let $N=\torus^{m-1} \times \mathbf{R}$ be the warped product of a flat torus $\torus^{m-1}$ and a real line with metric $g_N= \cosh^{\frac{2}{a}}(at) g_{\torus^{m-1}} + dt^2$, where $a\in [\frac{m-1}{2},\frac{m}{2}]$ (cf.~\cite{MWang24}*{Example~1.5}). The scalar curvature of $g_N$ is $\scal_{g_N}= -m(m-1)+(m-1)(m-2a)\cosh^{-2}(at)\in [-m(m-1), -2a(m-1)]$ and $\lambda_1(N, g_N)=\frac{(m-1)^2}{4}$. 
    Let $\mathbf{S}^1$ denote the standard circle and let $p\in \sph^1$ be a base-point. Choose a smooth map $\varphi:\mathbf{R}\to \sph^1$ of degree one such that $\varphi$ maps $\mathbf{R}\setminus (-1,1)$ to $p$.
    For any $\delta>0$, we can find a finite covering $\widehat{\torus}^{m-1} \to \torus^{m-1}$ together with a $\delta$-Lipschitz map $h: \widehat{\torus}^{m-1} \to \sph^{m-1}$ of nonzero degree.
    Let $M= \widehat{\torus}^{m-1} \times \mathbf{R}$ with the lifted metric $g_M:= \cosh^{\frac{2}{a}}(at) g_{\widehat{\torus}^{m-1}} + dt^2$. Again, $\scal_{g_M} = -m(m-1)+(m-1)(m-2a)\cosh^{-2}(at)$ and $\lambda_1(M,g_M)=\frac{(m-1)^2}{4}$.
    Consider the chain of maps
    \[
        M= \widehat{\torus}^{m-1} \times \mathbf{R} \overset{h\times \varphi}{\to} \sph^{m-1} \times \sph^1 \overset{\phi}{\to} \sph^m,
    \]
    where $\phi$ is a smooth map of degree one factoring through the smash product $\sph^{m-1} \wedge \sph^1$. Let $\Phi=\phi \circ (h\times \varphi)$. Then $\Phi$ is smooth and area $(\delta\cdot C)$-decreasing for some $C>0$ (that is, $|{\rm d}_p\Phi(v_1)\wedge {\rm d}_p\Phi(v_2)|\leq \delta\cdot C |v_1\wedge v_2|$ for all $p\in M$ and $v_1,v_2\in T_p M$), with $\deg(\Phi)=\deg(h)\neq 0$.
    For any $\varepsilon>0$, choosing $\delta$ sufficiently small makes $\Phi$ area $\varepsilon$-decreasing.

    By \cite{Shi25+}*{Lemma~4.5 and Remark~4.6}, with $K = \supp(\mathrm{d}\Phi)$ we have
    \[
   \Acw(M|K) \geq \frac{2}{\varepsilon}.
    \]
    Since this inequality holds for all $\varepsilon > 0$, letting $\varepsilon \to 0$ gives $\Acw(M|K) = \infty$. From \cite{Shi25+}*{Remark~3.6} we obtain $\Acw(M) \geq \Acw(M|K)$; hence $\Acw(M) = \infty$ and consequently $\glcw(M) = \infty$.
\end{example}

In~\cite{Ch75}, Cheng proved that for an $m$-dimensional complete manifold $(M,g_M)$ with $\Ric_{g_M} \geq (m-1)\kappa$ for some $\kappa \leq 0$, the bottom spectrum satisfies
\[
    \lambda_1(M,g_M) \leq - \frac{(m-1)^2}{4} \kappa.
\]
Theorem~\ref{mthm:codimension_zero} provides a bridge to extend Cheng's result to the setting of a (possibly positive) scalar curvature lower bound.

\begin{theorem} \label{thm:bottom_spectrum_all_dimension}
    Let $(M,g_M)$ be an $m$-dimensional complete Riemannian spin manifold such that $\scalM\geq \kappa$ for some constant $\kappa \leq \frac{2m(m-1)}{\glcw(M)}$. Then
    \[
    \lambda_1(M,g_M) \leq  \frac{m-1}{4m} \bigg( \frac{2m(m-1)}{\glcw(M)} - \kappa \bigg).
    \]
\end{theorem}

Since every oriented three-manifold is spin, we obtain the following estimate for the bottom spectrum under a scalar curvature lower bound.
\begin{corollary}\label{cor:three_bottom_spectrum}
    Let $(M,g_M)$ be a three-dimensional complete Riemannian manifold with $\scalM\geq \kappa$ for some constant $\kappa\leq 0$. Suppose $\glcw(M)=\infty$. Then
    \begin{equation}\label{eq:bottom_spectrum_estimate_in_three_dimension}
         \lambda_1(M,g_M) \leq - \frac{1}{6} \kappa.
    \end{equation}
\end{corollary}

The following example, inspired by~\cite{MWang24}*{Example~1.2}, demonstrates that the condition $\glcw(M)=\infty$ is essential.

\begin{example}
    Consider any three-dimensional manifold
    \[
    N=X_1\# \cdots \# X_{\ell} \# (\sph^2\times \sph^1) \# \cdots \# (\sph^2\times \sph^1),
    \]
    where each $X_j$ is diffeomorphic to $\sph^3/\Gamma_j$ for some $\Gamma_j\subset O(4)$ acting standardly with $|\Gamma_j|\geq 3$. 
    By~\cite{GL80}*{Theorem~5.4} and~\cite{GL83}*{Theorem~8.1}, $N$ admits a metric $g_N$ of positive scalar curvature, and thus the lifted metric $g_M$ on its universal cover $M$ also has positive scalar curvature.
    Since $\pi_1(N)$ contains a free subgroup $F_2$, it is non-amenable. Therefore, $\lambda_1(M, g_M)>0$ by~\cite{Brooks81}*{Theorem~1}. By Corollary~\ref{cor:complete_manifold_NNSC_mean_convex_boundary}, we see that $\glcw(M)<\infty$.
\end{example}

\begin{remark}
    Munteanu--Wang~\cite{MWang24}*{Theorem~1.1} proved \eqref{eq:bottom_spectrum_estimate_in_three_dimension} in dimension three under alternative topological assumptions and a lower Ricci curvature bound.
\end{remark}

A particularly interesting special case of Theorem~\ref{thm:bottom_spectrum_all_dimension} is the following.
\begin{corollary}\label{cor:sectional_bottom_spectrum}
    Let $(M,g_M)$ be an $m$-dimensional complete Riemannian manifold with $\Sec_{g_M}\leq 0$ and $\scalM\geq \kappa$ for some $\kappa \leq 0$. Then
    \begin{equation}\label{eq:bottom_spectrum_estimate_Cartan-Hadamard}
         \lambda_1(M,g_M)\leq -\frac{m-1}{4m} \kappa.
    \end{equation}
\end{corollary}

\begin{remark}
    The three-dimensional case of \eqref{eq:bottom_spectrum_estimate_Cartan-Hadamard} was established by Munteanu--Wang~\cite{MWang23}*{Corollary 1.4}. Wang--Zhu~\cite{WZhu2024+}*{Corollary~4.9} proved the same estimate under the additional assumption of bounded geometry. Our result removes these extra hypotheses.
\end{remark}

For manifolds with mean-convex boundary, we prove the following counterpart of Theorem~\ref{mthm:codimension_zero}.

\begin{theorem}\label{mthm:codimension_zero_boundary}
    Let $(M,g_M)$ be an $m$-dimensional complete (possibly noncompact) Riemannian spin manifold with compact mean-convex boundary. Then
    \[
        \inf_M\scalM\leq \frac{2m(m-1)}{\glcw(M)}.
    \]
\end{theorem}

\vspace{.3mm}

\textbf{Notation.} Unless otherwise specified, all manifolds are assumed smooth, connected, oriented, and of dimension at least two. Completeness refers to metric completeness when $M$ has nonempty boundary.

\textbf{Organization of the paper.} Section~\ref{sec:Callias_operator_&_spectral_flow} provides the necessary technical background on Callias operators and spectral flow. Section~\ref{sec:formula_in_codimension_zero} contains the proof of Theorem~\ref{mthm:codimension_zero}. Section~\ref{sec:proof_boundary} proves Theorem~\ref{mthm:codimension_zero_boundary}.

This paper is based on part of the work~\cite{Liu26b+}. For improved readability, the original text has been split into two separate articles.
After our preprint was posted, Stryker~\cite{Stryker26+} employed a special case of \eqref{eq:spectral_cowaist} to obtain an index-theoretic proof of the classification of complete two-sided stable minimal surfaces in $\mathbf{R}^3$.

\section{Preliminaries}\label{sec:Callias_operator_&_spectral_flow}

This section reviews the fundamental properties of deformed Dirac (Callias) operators. Let $(M,g_M)$ be an $m$-dimensional complete Riemannian manifold, possibly noncompact and with compact boundary.

\subsection{Callias operators and spectral flow}
When $M$ is spin, every GL‑pair on $M$ canonically determines a \textit{relative Dirac bundle} in the sense of~\cite{CZ24}*{Definition 2.2}. Set
\[
S:=\slashed{S}_M \,\widehat{\otimes}\, (\mathcal{E}\oplus \mathcal{F}^{{\rm op}}),
\]
where $\slashed{S}_M$ is the complex spinor bundle of $M$, $\widehat{\otimes}$ denotes the graded tensor product, and $\mathcal{E}\oplus \mathcal{F}^{{\rm op}}$ is the $\mathbf{Z}_2$-graded bundle with $\mathcal{E}$ even and $\mathcal{F}$ odd. Thus $S=S^{+}\oplus S^{-}$, where $S^{+}=(\slashed{S}_M\otimes\mathcal{E})\oplus(\slashed{S}_M\otimes\mathcal{F})$ and $S^{-}=(\slashed{S}_M\otimes\mathcal{F})\oplus(\slashed{S}_M\otimes\mathcal{E})$.
Outside a compact set $K$, one defines an odd, self‑adjoint, parallel bundle involution
\[
\sigma:=\operatorname{id}_{\slashed{S}_M}\,\widehat{\otimes}\,
\begin{pmatrix}
0 & \mathcal{I}^{*} \\
\mathcal{I} & 0
\end{pmatrix}
\;\colon\; S|_{M\setminus K}\longrightarrow S|_{M\setminus K}.
\]
Let $\mathcal{D}$ be the Dirac operator on $S$ (see~\cite{CZ24}*{Example 2.5}).
A Lipschitz function $f: M\to \mathbf{R}$ is called an \textit{admissible potential} if $f=0$ on $K$ and there exists a compact set $K\subseteq L\subseteq M$ such that $f$ is equal to a nonzero constant on each component of $M\setminus L$ (cf.~\cite{CZ24}*{Definition 3.1}). The product $f\sigma$ then extends by zero to a continuous bundle map on $M$.

\begin{definition}
    The \textit{Callias operator} associated to the above data is
    \[
    \mathcal{B}_{f} := \mathcal{D} + f \sigma.
    \]  
\end{definition}

For a manifold $M$ with compact boundary, we impose local boundary conditions via a boundary chirality.

\begin{definition}[\cite{CZ24}]
    Let $S$ be a relative Dirac bundle and $s\colon \partial M\to \{\pm 1\}$ a locally constant function. The \textit{boundary chirality} associated to $s$ is the endomorphism
    \[
         \chi := s\, c(\nu^*)\sigma \colon \left. S\right|_{\partial M} \to \left. S\right|_{\partial M},
    \]
    where $\nu^*$ is the dual covector of the outward unit normal $\nu$ to $\partial M$.
\end{definition}

The map $\chi$ is a self‑adjoint, even involution; it anti‑commutes with $c(\nu^*)$ and commutes with $c(w^*)$ for all $w\in T(\partial M)$. This leads to the following boundary condition.

\begin{definition}[\cite{CZ24}]
    A section $u\in C^{\infty}(M,S)$ satisfies the \textit{local boundary condition} if $\chi \left(u|_{\partial M}\right) = u|_{\partial M}$.
\end{definition}

For a choice of $s:\partial M\to \{\pm 1\}$, denote by $\mathcal{B}_{f,s}$ the restriction of $\mathcal{B}_{f}$ to
\[
   {\rm dom}(\mathcal{B}_{f,s}):=\{ u\in C_c^{\infty}(M,S)\colon \chi(u|_{\partial M})=u|_{\partial M} \}.
\]

Since $S=S^+\oplus S^-$ is $\mathbf{Z}_2$-graded and $\chi$ is even, $\mathcal{B}_{f,s}$ splits as $\mathcal{B}_{f,s}=\mathcal{B}_{f,s}^{+}\oplus\mathcal{B}_{f,s}^{-}$ with
\[
   \mathcal{B}_{f,s}^{\pm}: \{u\in C_c^{\infty}(M,S^{\pm}) \colon \chi(u|_{\partial M}) = u|_{\partial M}\} \to L^2(M,S^{\mp}).
\]
By~\cite{CZ24}*{Theorem~3.4}, $\mathcal{B}_{f,s}$ is self-adjoint and Fredholm. Its \textit{index} is 
\[
\ind(\mathcal{B}_{f,s}):=\dim (\ker(\mathcal{B}_{f,s}^+)) - \dim (\ker(\mathcal{B}_{f,s}^-)).
\]

Assume $\dim M$ is odd and let $\rho_M\in C^\infty(M,U(l))$ be a smooth map such that $[\mathcal{D},\rho_M]$ is bounded on $\operatorname{dom}(\mathcal{B}_{f,s})$. Following~\cite{Ge93}*{Section~1, p.~491}, consider the trivial bundle $\mathcal{E}_0 := M \times \mathbf{C}^l$ with connection $\D$. Then $\rho_M$ determines a family of Hermitian connections
\[
    \nabla^{\mathcal{E}_0}(t):=\D + t \rho_M^{-1}[\D, \rho_M], \quad t\in [0,1],
\]
with curvature $R^{\nabla^{\mathcal{E}_0}(t)}= -t(1-t)(\rho_M^{-1}(\D\rho_M))^2$.

Recall that $R^{\rho_M}$ is defined in~\eqref{eq:defn_rho_curvature}. Note that $\| R^{\nabla^{\mathcal{E}_0}(t)} \|_{\infty} \leq \|R^{\rho_M}\|_{\infty}$ for all $t\in [0,1]$.

Using this family of connections, we construct a family of Callias operators. On the twisted bundle $S\otimes\mathbf{C}^{l}$, set
\[
\nabla(t) := \nabla^{S}\otimes\operatorname{id} + \operatorname{id}\otimes\nabla^{\mathcal{E}_0}(t), \quad t\in[0,1],
\]
and let $\mathcal{D}(t)$ be the Dirac operator associated with $\nabla(t)$; explicitly,
\[
\mathcal{D}(t) = \mathcal{D} + t\,\rho_M^{-1}[\mathcal{D},\rho_M] =(1-t)\mathcal{D}+ t\rho_M^{-1} \mathcal{D} \rho_M.
\]

The involution $\sigma$ extends naturally to $S\otimes\mathbf{C}^{l}$, and we identify $S\otimes\mathbf{C}^{l}$ with $S$ for simplicity. Since $\rho_M$ preserves $\operatorname{dom}(\mathcal{B}_{f,s})$ and commutes with $f$ and $\sigma$, for $t\in[0,1]$ we define a family of Callias operators on $\operatorname{dom}(\mathcal{B}_{f,s})$ by
\[
\mathcal{B}_{f}(t) := (1-t)\mathcal{B}_{f} + t\,\rho_M^{-1}\mathcal{B}_{f}\rho_M = \mathcal{D}(t) + f\sigma.
\]
Then $\{\mathcal{B}_{f,s}(t)\}_{t\in[0,1]}$ is a continuous family of self‑adjoint Fredholm operators in the Riesz topology (cf.~\cite{BMJ05}, \cite{Lesch05}).

\begin{definition}[{\cite{Shi24+}}]
    The \textit{spectral flow} of $\{\mathcal{B}_{f,s}(t)\}_{t\in[0,1]}$, denoted $\sflow(\mathcal{B}_{f,s},\rho_M)$, is the net number of eigenvalues that change from negative to nonnegative as $t$ increases from $0$ to $1$.
\end{definition}

\begin{remark}
    If every $\mathcal{B}_{f,s}(t)$ is invertible, the spectral flow $\sflow(\mathcal{B}_{f,s},\rho_M)$ necessarily vanishes.
\end{remark}

\subsection{Relative index and relative spectral flow}

Let $(\mathcal{E},\mathcal{F})$ be a GL-pair on $M$ with support $K$. Choose a smooth compact spin submanifold $\mathcal{W}$ containing $\partial M$ such that $\operatorname{int}(\mathcal{W})$ contains $K$. Denote by $\mathcal{W}^-$ a copy of $\mathcal{W}$ with the opposite orientation. The double $\mathrm{d}\mathcal{W}:=\mathcal{W}\cup_{\partial\mathcal{W}}\mathcal{W}^-$ carries a natural spin structure. On $\mathrm{d}\mathcal{W}$, define the \textit{virtual bundle} $V(\mathcal{E},\mathcal{F})$, which coincides with $\mathcal{E}$ over $\mathcal{W}$ and with $\mathcal{F}$ over $\mathcal{W}^-$.

If $\dim M$ is even, the \textit{relative index} of $(\mathcal{E},\mathcal{F})$ is defined as (see~\cite{CZ24})
\begin{equation}\label{eq:defn_relative_index}
    \indrel(M; \mathcal{E},\mathcal{F}) := \ind( \slashed{D}_{{\rm d}\mathcal{W},V(\mathcal{E},\mathcal{F})} ).
\end{equation}

If $\dim M$ is odd, let $\rho=\rho^+\oplus \rho^-\in C^{\infty}(\mathcal{W},U(l)\oplus U(l))$ with $\rho^+=\rho^-$ locally constant near $\partial\mathcal{W}$. Gluing $\rho^+$ and $\rho^-$ across the boundary gives a smooth map $\widetilde{\rho}\in C^{\infty}(\mathrm{d}\mathcal{W},U(l))$. Let $\rho_M$ denote the trivial extension of $\rho$ to $M$. The \textit{relative spectral flow} (cf.~\cite{Liu25a+}) of $(\mathcal{E},\mathcal{F},\rho_M)$ is
\begin{equation}\label{eq:defn_relative_spectral_flow}
    \sfrel(M; \mathcal{E},\mathcal{F}, \rho_M) := \sflow( \slashed{D}_{{\rm d}\mathcal{W},V(\mathcal{E},\mathcal{F})},\widetilde{\rho}).
\end{equation}

\section{Proof of Theorem~\ref{mthm:codimension_zero}}\label{sec:formula_in_codimension_zero}

We begin with the core estimate, which implies the main theorem.

\begin{proposition}\label{pro:key_proposition_odd}
    Let $(M,g_M)$ be an $m$-dimensional complete (possibly noncompact) Riemannian spin manifold. If $c\in (\frac{m-1}{4m},\infty)$, then
    \begin{equation}\label{eq:c-spectral_GL-cowiast_odd}
       \inf_M \scalM + c^{-1} \lambda_1(M,g_M) \leq  \frac{2m(m-1)}{\glcw(M)}.
    \end{equation}
\end{proposition}

\begin{proof}
    We split the argument into even and odd dimensions.

    \medskip
    \textbf{Case 1: $m$ is even.}
    If there are no $\widehat{A}$-admissible GL-pairs, then $\glcw(M)=0$ and the statement is trivial. Otherwise, let $(\mathcal{E},\mathcal{F})$ be an $\widehat{A}$-admissible GL-pair satisfying
    \begin{equation}\label{eq:relative_topo_obst}
      \int_{M} \Af(M) \wedge (\ch(\mathcal{E})- \ch(\mathcal{F})) \neq 0.
    \end{equation}

    Let $\mathcal{W}$ be a compact submanifold of $M$ with smooth boundary, whose interior contains the support $K$ of $(\mathcal{E},\mathcal{F})$, and let $\mathcal{U} \subset \operatorname{int}(\mathcal{W})$ be an open neighborhood of $K$. Let $S$ be the relative Dirac bundle associated with $(\mathcal{E},\mathcal{F})$ and $\mathcal{D}$ the corresponding Dirac operator (cf.~\cite{CZ24}*{Example~2.5}).

    Let $\psi:M\to [0,1]$ be smooth with $\psi=0$ on $K$ and $\psi=1$ on $M\setminus \overline{\mathcal{U}}$. As in \cite{Zh20}, for $\varepsilon>0$, set $f:=\varepsilon \psi$. Then $f$ is an admissible potential function vanishing on $K$ and equal to $\varepsilon$ outside $\overline{\mathcal{U}}$.

    Consider $\mathcal{B}_{f}=\mathcal{D}+f\sigma$. By the splitting theorem for the index of Callias operators~\cite{Rad94}*{Proposition~2.3} (see also~\cite{CZ24}*{Theorem~3.6}),
    \begin{equation}\label{eq:to_splitting}
        \ind(\mathcal{B}_{f}) = \ind(\mathcal{B}_{f,-1}^{\mathcal{W}}) + \ind(\mathcal{B}_{f,-1}^{ \overline{M\setminus \mathcal{W}} }), 
    \end{equation}
    where $\mathcal{B}_{f,-1}^{\mathcal{W}}$ and $\mathcal{B}_{f,-1}^{ \overline{M\setminus \mathcal{W}} }$ are the corresponding Callias operators on $\mathcal{W}$ and $\overline{M\setminus \mathcal{W}}$, respectively.

    Since $\mathcal{B}_{f,-1}^{ \overline{M\setminus \mathcal{W}} }$ has empty support and $s=-1$ on all of $\partial (\overline{M\setminus \mathcal{W}})$,~\cite{CZ24}*{Lemma~3.7} gives
    \begin{equation}\label{eq:vanishing_to_at_infinity}
        \ind(\mathcal{B}_{f,-1}^{ \overline{M\setminus \mathcal{W}} }) = 0.
    \end{equation}
    By~\cite{CZ24}*{Corollary~3.9},
    \begin{equation}\label{eq:correspond}
        \ind(\mathcal{B}_{f,-1}^{\mathcal{W}}) = \indrel(M;\mathcal{E},\mathcal{F}).
    \end{equation}
    Using \eqref{eq:defn_relative_index}, \eqref{eq:relative_topo_obst}, and the cohomological formula for the index (cf.~\cite{AS63}, see also~\cite{LM89}*{p.~256, (13.26)}),
    \begin{equation}\label{eq:cohomological_formula_relative_index}
        \begin{aligned}
            \indrel(M;\mathcal{E},\mathcal{F}) 
            &= \int_{{\rm d}\mathcal{W}} \Af({\rm d}\mathcal{W}) \wedge \ch( V(\mathcal{E},\mathcal{F}) ) \\
            &=\int_{\mathcal{W}} \Af(\mathcal{W}) \wedge ( \ch(\mathcal{E})-\ch(\mathcal{F}) ) \\
            &=\int_{M} \Af(M) \wedge ( \ch(\mathcal{E})-\ch(\mathcal{F}) ) \neq 0,
        \end{aligned}
    \end{equation}
    where the second line uses the opposite orientation on $\mathcal{W}^-$, and the third line holds because $\mathcal{E}$ and $\mathcal{F}$ are identified outside $K \subset \mathcal{W}$.

    Combining \eqref{eq:to_splitting}--\eqref{eq:cohomological_formula_relative_index} yields $\ind(\mathcal{B}_{f}) \neq 0$. Hence there exists a nonzero $u\in \ker(\mathcal{B}_{f})$. Because $\partial M = \emptyset$, the boundary term in the spectral estimate~\cite{Liu25a+}*{(1.18)} vanishes. Using~\cite{Liu25a+}*{(1.18)} and~\cite{CZ24}*{(2.18)}, we obtain
    \[
    \begin{aligned}
        0 \geq & \frac{m}{m-1} \int_M \left( \frac{1}{4c} |{\rm d} |u||^2 + \frac{1}{4} \scalM |u|^2 + \langle u, \mathscr{R}^{\mathcal{E} \oplus \mathcal{F}} u \rangle \right) dV \\
        & + \int_M  \left \langle u, \alpha_2 f^2 u +  \alpha_2 c({\rm d} f) \sigma u \right \rangle  dV,
    \end{aligned}
    \]
    where $c>\frac{m-1}{4m}$ and $\alpha_2>0$ are constants, and
    \[
    \mathscr{R}^{\mathcal{E} \oplus \mathcal{F}} = \sum_{i<j} c(e^i)c(e^j) ({\rm id}_{\slashed{S}_M}\otimes R^{\mathcal{E} \oplus \mathcal{F}}(e_i, e_j))
    \]
    denotes the curvature endomorphism. Since $\mathscr{R}^{\mathcal{E} \oplus \mathcal{F}}$ depends linearly on the curvature, we have the pointwise bounds
    \begin{equation}\label{eq:curvature_term_bound_index_case}
        \langle u, \mathscr{R}^{\mathcal{E}\oplus \mathcal{F}} u\rangle \geq -\frac{m(m-1)}{2} \|R^{\mathcal{E}\oplus \mathcal{F}}\|_{\infty} \, |u|^2 \quad \text{and}\quad \left \langle u, c({\rm d} f) \sigma u \right \rangle \geq -|{\rm d}f| \, |u|^2,
    \end{equation}
    hence
    \begin{align*}
           0 \geq & \frac{m}{4c(m-1)} \int_M |{\rm d} |u||^2 dV + \frac{m}{m-1} \int_M \left( \frac{1}{4} \scalM |u|^2 -\frac{m(m-1)}{2} \|R^{\mathcal{E}\oplus \mathcal{F}}\|_{\infty} \, |u|^2 \right) dV \\
        & + \alpha_2 \int_M  (f^2-|{\rm d}f|)|u|^2 dV.
    \end{align*}

    Recall $f=\varepsilon \psi$ and $\supp({\rm d}\psi)\subset \overline{\mathcal{U}}\setminus K$. By \eqref{defn:bottom_spectrum},
    \begin{align*}
        0 \geq & \Big( \frac{m}{4c(m-1)} \lambda_1(M,g_M) + \frac{m}{4(m-1)} \inf_M \scalM \\
               & \quad - \frac{m^2}{2} \|R^{\mathcal{E}\oplus \mathcal{F} } \|_{\infty} - \varepsilon \alpha_2 \sup_{ \overline{\mathcal{U}}\setminus K } |{\rm d} \psi| \Big)  \|u\|_{L^2(M)}^2,
    \end{align*}
    where $\|u \|_{L^2(M)}^2 = \int_M |u|^2 dV$. Since $\|u \|_{L^2(M)}^2>0$, dividing by this quantity and letting $\varepsilon\to 0$ gives
    \[
      \inf_M \scalM + c^{-1} \lambda_1(M,g_M) \leq 2m(m-1) \|R^{\mathcal{E}\oplus \mathcal{F} } \|_{\infty}.
    \]
    Taking the infimum over all $\widehat{A}$-admissible GL-pairs completes the even case.

    \medskip
    \textbf{Case 2: $m$ is odd.}
    If there are no $\widehat{A}$-admissible GL-triples, then $\glcw(M)=0$ and the statement is trivial. Otherwise, let $(\mathcal{E},\mathcal{F},\rho_{M})$ be such a triple satisfying
    \begin{equation}\label{eq:relative_topo_obst_odd}
      \int_{M} \Af(M) \wedge (\ch(\mathcal{E})- \ch(\mathcal{F})) \wedge \ch(\rho_{M}) \neq 0.
    \end{equation}

    There exists a compact $K\subset M$ outside which the triple is trivial. Let $\mathcal{W}$ be a compact submanifold with smooth boundary whose interior contains $K$, and $\mathcal{U} \subset \operatorname{int}(\mathcal{W})$ an open neighborhood of $K$. Denote by $\rho^+$ and $\rho^-$ the restrictions of $\rho_M$ to $\mathcal{W}$ and its opposite $\mathcal{W}^-$. Then $\rho^+=\rho^-$ near $\partial\mathcal{W}$, and they glue to $\widetilde{\rho}\in C^{\infty}({\rm d}\mathcal{W},U(l))$. Set $\rho=\rho^+\oplus \rho^-$.

    Let $S$ be the relative Dirac bundle associated with $(\mathcal{E},\mathcal{F})$ and $\mathcal{D}$ the Dirac operator (see~\cite{CZ24}*{Example~2.5}). Let $f$ be the admissible potential from Case 1. Consider the family $\mathcal{B}_{f}(t)=(1-t)\mathcal{B}_{f} + t \rho_M^{-1} \mathcal{B}_{f} \rho_M$ for $t\in [0,1]$. By the splitting theorem for spectral flow~\cite{Shi24+}*{Theorem~2.10},
    \begin{equation}\label{eq:to_splitting_odd}
        \sflow(\mathcal{B}_{f},\rho_M) = \sflow(\mathcal{B}_{f,-1}^{\mathcal{W}},\rho) + \sflow(\mathcal{B}_{f,-1}^{ \overline{M\setminus \mathcal{W}} },\rho_{ \overline{M\setminus \mathcal{W}} } ).
    \end{equation}

    As in Case 1, the second term vanishes by~\cite{Shi24+}*{Lemma~3.4}. Using~\cite{Shi24+}*{Theorem~3.10},
    \begin{equation}\label{eq:correspond_odd}
        \sflow(\mathcal{B}_{f,-1}^{\mathcal{W}},\rho) = \sfrel(M;\mathcal{E},\mathcal{F},\rho_M).
    \end{equation}
    By the cohomological formula for spectral flow~\cite{Ge93}*{Theorem~2.8},
    \begin{equation}\label{eq:cohomological_formula_relative_sflow}
        \begin{aligned}
            \sfrel(M;\mathcal{E},\mathcal{F},\rho_M) 
            &= \int_{{\rm d}\mathcal{W}} \Af({\rm d}\mathcal{W}) \wedge \ch(V(\mathcal{E},\mathcal{F})) \wedge \ch(\widetilde{\rho}) \\
            &=\int_{\mathcal{W}} \Af(\mathcal{W}) \wedge ( \ch(\mathcal{E})-\ch(\mathcal{F}) ) \wedge \ch(\rho^+) \\
            &=\int_M \Af(M) \wedge ( \ch(\mathcal{E})-\ch(\mathcal{F}) ) \wedge \ch(\rho_M) \neq 0,
        \end{aligned}
    \end{equation}
    where we used that $\mathcal{E}\cong\mathcal{F}$ and $\rho^+=\rho^-$ near $\partial\mathcal{W}$, and that the triple is trivial outside $K\subset\mathcal{W}$.

    Combining \eqref{eq:to_splitting_odd}--\eqref{eq:cohomological_formula_relative_sflow} gives $\sflow(\mathcal{B}_{f},\rho_M) \neq 0$. Hence for some $t\in [0,1]$ there exists a nonzero $u\in \ker(\mathcal{B}_{f}(t))$. Since $\partial M=\emptyset$, the boundary term in~\cite{Liu25a+}*{(1.17)} vanishes. Using~\cite{Liu25a+}*{(1.17)} and~\cite{Shi25+}*{(4.1)},
    \[
    \begin{aligned}
        0 \geq & \frac{m}{m-1} \int_M \left( \frac{1}{4c} |{\rm d} |u||^2 + \frac{1}{4} \scalM |u|^2 + \langle u, \mathscr{R}^{\mathcal{E} \oplus \mathcal{F}} u - 4t(1-t) \mathscr{R}^{\rho_M} u \rangle \right) dV \\
        & + \int_M  \left \langle u, \alpha_2 f^2 u +  \alpha_2 c({\rm d}f) \sigma u \right \rangle  dV,
    \end{aligned}
    \]
    where $c>\frac{m-1}{4m}$, $\alpha_2>0$, and $\mathscr{R}^{\rho_M}=\sum_{i<j} c(e^i)c(e^j)R^{\rho_M}(e_i,e_j)$. The pointwise estimate
    \[
        \langle u, \mathscr{R}^{\mathcal{E} \oplus \mathcal{F}} u - 4t(1-t) \mathscr{R}^{\rho_M} u \rangle \geq -\frac{m(m-1)}{2} (\|R^{\mathcal{E}\oplus\mathcal{F}}\|_{\infty} + \|R^{\rho_M}\|_{\infty}) |u|^2
    \]
    leads, after the same manipulations as in Case 1, to
    \[
      \inf_M \scalM + c^{-1} \lambda_1(M,g_M) \leq 2m(m-1) (\|R^{\mathcal{E}\oplus \mathcal{F} } \|_{\infty} + \|R^{\rho_M}\|_{\infty}).
    \]
    Taking the infimum over all $\widehat{A}$-admissible GL-triples completes the odd case.
\end{proof}

\begin{proof}[Proof of Theorem~\ref{mthm:codimension_zero}]
    By Proposition~\ref{pro:key_proposition_odd}, for any $\epsilon>0$,
    \[
    \inf_M \scalM + \Big(\frac{m-1}{4m}+\epsilon\Big)^{-1} \lambda_1(M,g_M) \leq \frac{2m(m-1)}{\glcw(M)}.
    \]
    Letting $\epsilon \to 0^+$ yields the desired inequality.
\end{proof}

Corollaries~\ref{cor:complete_manifold_NNSC_mean_convex_boundary} and~\ref{thm:bottom_spectrum_all_dimension} follow immediately from Theorem~\ref{mthm:codimension_zero}.

\begin{proof}[Proof of Corollary~\ref{cor:sectional_bottom_spectrum}]
    This is a consequence of Theorem~\ref{thm:bottom_spectrum_all_dimension} and the Cartan--Hadamard theorem.
\end{proof}

\section{Proof of Theorem~\ref{mthm:codimension_zero_boundary}}\label{sec:proof_boundary}

We now adapt the previous argument to the boundary case.

\begin{proof}[Proof of Theorem~\ref{mthm:codimension_zero_boundary}]
    We treat the even-dimensional case; the odd case is analogous and omitted. Let $(\mathcal{E},\mathcal{F})$ be an $\widehat{A}$-admissible GL-pair. Let $\mathcal{W}$ be a compact submanifold containing $\partial M$ with smooth boundary, whose interior contains the support $K$, and let $\mathcal{U} \subset \operatorname{int}(\mathcal{W})$ be an open neighborhood of $K$.

    Let $S$ be the relative Dirac bundle associated with $(\mathcal{E},\mathcal{F})$ and $\mathcal{D}$ the Dirac operator (cf.~\cite{CZ24}*{Example~2.5}). Choose a smooth cut-off function $\psi: M\to [0,1]$ with $\psi=0$ on $K$ and $\psi=1$ outside $\overline{\mathcal{U}}$. For $\varepsilon>0$, set $f:=\varepsilon\psi$; this is an admissible potential.

    Consider the Callias operator $\mathcal{B}_{f,-1}$ with sign $s=-1$. As in the proof of Proposition~\ref{pro:key_proposition_odd}, one checks that
    \[
        \ind(\mathcal{B}_{f,-1}) = \int_M \Af(M)\wedge (\ch(\mathcal{E})-\ch(\mathcal{F})) \neq 0,
    \]
    so there exists a nonzero $u\in \ker(\mathcal{B}_{f,-1})$. From~\cite{Liu25a+}*{(1.18)} and~\cite{CZ24}*{(2.18)},
    \[
    \begin{aligned}
    0 \geq & \frac{m}{4c(m-1)} \int_M| {\rm d} |u| |^2 dV + \frac{m}{m-1} \int_M \left( \frac{1}{4} \scalM |u|^2 + \langle u, \mathscr{R}^{\mathcal{E} \oplus \mathcal{F}} u \rangle \right) dV \\
      &  + \int_M \langle u, \alpha_2 f^2 u + \alpha_2  c({\rm d}f) \sigma u \rangle dV + \int_{\partial M} 
      \underbrace{ \big( \alpha_2 f + \frac{m}{2} H \big) }_{\geq 0 \text{ (since } f, H \geq 0\text{)}} |u|^2 dA,
    \end{aligned}
    \]
    where $c>\frac{m-1}{4m}$ and $\alpha_2>0$. Dropping the nonnegative gradient and boundary terms and using \eqref{eq:curvature_term_bound_index_case} yields
    \[
       0 \geq \frac{m}{m-1} \int_M \left( \frac{1}{4} \scalM |u|^2 -\frac{m(m-1)}{2} \|R^{\mathcal{E}\oplus \mathcal{F}}\|_{\infty} \, |u|^2 \right) dV + \alpha_2 \int_M (f^2-|{\rm d}f|)|u|^2 dV.
    \]

    Since $f=\varepsilon \psi$ and $\supp({\rm d}\psi)\subset \overline{\mathcal{U}}\setminus K$, the same limiting argument as before leads to
    \[
        \inf_M \scalM \leq 2m(m-1) \|R^{\mathcal{E}\oplus\mathcal{F}}\|_{\infty}.
    \]
    Taking the infimum over all $\widehat{A}$-admissible GL-pairs gives the desired inequality.
\end{proof}

\textbf{Acknowledgements.} 
The author is deeply grateful to Professor Weiping Zhang for his insightful discussions and suggestions. He also thanks Professor Guangxiang Su for his valuable comments, and Professors Zhenlei Zhang and Bo Liu for their continuous encouragement. This work is partially supported by the National Natural Science Foundation of China (Grant No. 12501064), the China Postdoctoral Science Foundation (Grant Nos. 2025M773075, GZC20252016, and 2025T002TJ), and the Nankai Zhide Foundation.


\end{document}